\documentclass[11pt]{article}
\usepackage[spanish,english,activeacute]{babel}
\usepackage{amssymb,url}
\usepackage{amsfonts,amsmath,mathrsfs}
\usepackage{amsthm}
\usepackage{latexsym}
\usepackage{enumerate}
\newtheorem{teo}{Theorem}
\newtheorem{lema}{Lemma}
\newtheorem{coro}{Corollary}

\newtheorem{defi}{Definition}
\begin{document}
\title{The nontrivial zeros of the Zeta Function lie on the Critical Line}
\author{Pedro Geraldo\footnote{In Memorian to: G. F. B. Riemann (1826 - 1866).}\\
{\url{(pegeraldo@luz.edu.ve)}}\\
{\small Departamento de Matem\'aticas}\\
{\small Facultad de Ingenier\'ia}\\
{\small Universidad del Zulia}\\
{\small N\'ucleo COL}\\
{\small Cabimas (4013), Venezuela}
\date{10/24/2008}}
\maketitle
\selectlanguage{english}
\begin{abstract}
In this paper is stablished a characterization of the solutions of the equation:
$\zeta(z)=0$. Then such a characterization is used to give a proof for Riemann's Conjecture.\\

{\bf Classification Subject:} 30B40 \& 11M26\\

\end{abstract}

\section{The Riemann Zeta Function}

Let $t\in\mathbb{R}^+$ and $\log{t}$ be its real value, then:
\[\forall\,n\in\mathbb{Z}\wedge n\geq
1:\,\Big|\frac{1}{n^z}\Big|=\frac{1}{n^{\mathrm{Re}z}}=\frac{1}{e^{
\mathrm{Re}z\cdot\log{n}}}\]
is a well defined function for every $z\in\mathbb{C}$.

Let $\delta>0$ be an arbitrary real number. For $\mathrm{Re}z\geq 1+\delta$, we have:
\[\Big|\frac{1}{n^z}\Big|=\frac{1}{n^{\mathrm{Re}z}}\le\frac{1}{n^{1+\delta}}\]

The $p$-series $\displaystyle\sum_{n=1}^\infty\frac{1}{n^{1+\delta}}$ is convergent and
Weirstrass
criterion says that the series $\displaystyle\sum_{n=1}^\infty\frac{1}{n^z}$ is also
absolutely
convergent for
$\mathrm{Re}z>1$.The Riemann Zeta Function is defined as follow:
\[\zeta(z)=\sum_{n=1}^\infty\frac{1}{n^z}\,\,\,\hbox{ for }\,\,\,\mathrm{Re}z>1\]
$\zeta$ is analytic in the half-plane $\mathrm{Re}z>1$ and uniformly convergent in every
compact set
contained in that half-plane $\mathrm{Re}z>1$.

\begin{defi}\label{def1} Let $E$ and $F$ be sets. Suppose $P\subset E$ and $g:P\to F$ be
an application. The application $f:E\to F$
is said to be an extension of $g$ over $E$ relative to $F$ if $f|_P=g$.
\end{defi}

In general such an application is not unique see \cite{8}.
However, any Analytic continuations (extension) if they exist are
unique, see \cite{15} and \cite{25}.

\begin{teo}\label{zetar} $\zeta$ can be continued across the boundary $\mathrm{Re}z=1$
of the half- plane $\mathrm{Re}>1$, and proves to be a Meromorphic
function having the single pole $z=1$ with the principal part
$\frac{1}{z-1}$; i.e., $z=1$ is a simple pole with  residue $+1$.
\end{teo}

\begin{proof} See \cite{17}
\end{proof}

From Theorem \ref{zetar}, we get that:
\begin{enumerate}
    \item[$(i)$] The analytic Continuation of $\zeta$ up to the boundary
    $\mathrm{Re}z=0$ is given
    by
    \begin{equation}\label{eq1}
    \zeta(z)=1+\frac{1}{z-1}-z\sum_{n=1}^\infty\int_0^1\frac{tdt}{(n+t)^{z+1}}
    \end{equation}
    \item[$(ii)$] The analytic Continuation of $\zeta$ up to $\mathrm{Re}z=-1$ is
    given by
    \[\zeta(z)=1+\frac{1}{z-1}-\frac{z}{2!}\Big[\zeta(z+1)-1\Big]-\frac{z(z+1)}{2!}
    \sum_{n=1}^\infty
    \int_0^1\frac{t^2dt}{(n+t)^{z+2}}\]
    \item[$(iii)$] The analytic Continuation of $\zeta$ up to $\mathrm{Re}z=-2$ is
    given by
    \begin{equation*}
    \begin{split}
    \zeta(z)=&1+\frac{1}{z-1}-\frac{z}{2!}\Big[\zeta(z+1)-1\Big]-\frac{z(z+1)}{3!}
    \Big[\zeta(z+2)-1
    \Big]\\
    &-\frac{z(z+1)(z+2)}{3!}\sum_{n=1}^\infty\int_0^1\frac{t^3dt}{(n+t)^{z+3}}
    \end{split}
    \end{equation*}
\end{enumerate}
and so forth by induction.

The most important issue  here is that by definition \ref{def1}, it is enough to proof
the Riemann Conjeture for
\[\zeta(z)=1+\frac{1}{z-1}-z\sum_{n=1}^\infty\int_0^1\frac{tdt}{(n+t)^{z+1}}\]
see also \cite{25} prop. 16.10.

Remember that: $\zeta(z)=0$ for $z=-2,-4,-6,\ldots$ which can be
deduced from Riemann's functional equation:
\begin{equation}\label{re}
\zeta(z)=2(2\pi)^{z-1}\Gamma(1-z)\zeta(1-z)\sin(\frac{1}{2}\pi z)|\text{ for }z\not=1
\end{equation}

$$-1<Re z<1$$
\newpage

We also know that $\zeta(0)\not=0$ and $\zeta(1)\not=0$, similar
reasoning gives that $\zeta$ has no other zeros outside the
Critical strip $\overline{B}$ than the trivials: $\{-2, -4,
\ldots\}$.

\begin{defi} The points $z=-2,-4,-6,\ldots$ are called the trivials zeros of $\zeta$.
\end{defi}

Let us define the following sets:
\begin{enumerate}
    \item[$(i)$] $F=\{z\in\mathbb{C}:\mathrm{Re}z=\frac{1}{2}\}$ called the critical
    line
    \item[$(ii)$] $B_1=\{z\in\mathbb{C}:0<\mathrm{Re}z<\frac{1}{2}\}$
    \item[$(iii)$] $B_2=\{z\in\mathbb{C}:\frac{1}{2}<\mathrm{Re}z<1\}$
    \item[$(iv)$] $B=B_1\cup B_2$
    \item[$(v)$] $\overline{B}=\{z\in\mathbb{C}:0\le \mathrm{Re}z\le 1\}$ called
    the critical strip. See more about this in \cite{11}.
\end{enumerate}

The Riemann Hypothesis is equivalent to say that $\zeta$ has no zeros in $B$.

\begin{lema}
If $z_0\in (\mathbb{C}\setminus\{0,1\})$, then:
$$z_0\sum_{n=1}^{\infty}\int_{0}^{1}\frac{tdt}{(n+t)^{z_0+1}}=1 \Rightarrow \zeta(z_0)\neq 0$$
\end{lema}
\begin{proof}
$$
\begin{array}{l}
\displaystyle \zeta(z_0)=1+\dfrac{1}{z_0-1}-z_0 \displaystyle{\sum_{n=1}^{\infty}\int_{0}^1\dfrac{tdt}{(n+t)^{z_0+1}}=1+\frac{1}{z_0-1}-1=\dfrac{1}{z_0-1}},\\\\
\displaystyle{\zeta(z_0)=\dfrac{1}{z_0-1}\Rightarrow (z_0-1)\zeta(z_0)=1 \Rightarrow \zeta(z_0)\neq 0}
\end{array}
$$
\end{proof}

\begin{coro}
$\displaystyle \zeta (z_0)=0 \Rightarrow z_0 \sum_{n=1}^{\infty}\int_{0}^{1}\frac{t dt}{(n+t)^{z_0+1}}\neq 1$
\end{coro}

\begin{teo}
If $z_0 \in (\mathbb{C}-\{0,1\})$. Then,
$$\zeta(z_0)=0 \Leftrightarrow (z_0-1)\sum_{n=1}^{\infty}\int_{0}^{1}\frac{tdt}{(n+t)
^{z_0+1}}=1$$
\end{teo}
\begin{proof}
\begin{itemize}
\item[(i)] 
$$
\begin{array}{rcl}
0 = \zeta(z_0) & \Rightarrow & \displaystyle{0=1+\frac{1}{z_0-1}-z_0\sum_{n=1}^{\infty}\int_0^1\frac{tdt}{(n+t)^{z_0+1}}}\\
 & = & \displaystyle{\frac{z_0-1+1}{z_0-1}-z_0\sum_{n=1}^{\infty}\int_{0}^{1}\frac{tdt}{(n+t)^{z_0+1}}}\\
 & = & \displaystyle{\frac{z_0}{z_0-1}-z_0\sum_{n=1}^{\infty}\int_{0}^{1}\frac{tdt}{(n+t)^{z_0+1}} = \frac{z_0-z_0(z_0-1)\displaystyle{\sum_{n=1}^{\infty}\int_{0}^{1}\frac{tdt}{(n+t)^{z_0-1}}}}{z_0-1}}\\
 & \Rightarrow & \displaystyle{0=\frac{z_0-z_0(z_0-1)\displaystyle{\sum_{n=1}^{\infty}\int_{0}^{1}\frac{tdt}{(n+t)^{z_0+1}}}}{z_0-1}}\\
 & \Rightarrow & \displaystyle{0\cdot (z_0-1)=z_0-z_0(z_0-1)\sum_{n=1}^{\infty}\int_{0}^1\frac{tdt}{(n+t)^{z_0+1}}}\\
 & \Rightarrow & \displaystyle{0=z_0\Big[1-(z_0-1)\sum_{n=1}^{\infty}\int_{0}^1\frac{tdt}{(n+t)^{z_0+1}}\Big]}\\
 & \Rightarrow & \displaystyle{1-(z_0-1)\sum_{n=1}^{\infty}\int_{0}^{1}\frac{tdt}{(n+t)^{z_0+1}}=0}\\
 & \Rightarrow & \displaystyle{(z_0-1)\sum_{n=1}^{\infty}\int_{0}^{1}\frac{tdt}{(n+t)^{z_0+t}}=1}.\\
\end{array}
$$

\item[(ii)] $\displaystyle (z_0-1)\sum_{1}^{\infty}\int_{0}^{1}\frac{tdt}{(n+t)^{z_0+1}}=1 \Rightarrow
\zeta(z_0)=0$
\end{itemize}
$$
\begin{array}{l}
\displaystyle{\zeta(z_0)=1+\frac{1}{z_0-1}-z_0\sum_{n=1}^{\infty}\frac{tdt}{(n+t)^{z_0+1}}}=\\\\
\displaystyle{(z_0-1)\sum_{1}^{\infty}\int_{0}^{1}\frac{tdt}{(n+t)^{z_0+1}}+\frac{1}{z_0-1}-z_0\sum_{n=1}^{\infty}\frac{tdt}{(n+t)^{z_0+1}}}=\\\\
\displaystyle{z_0\sum_{n=1}^{\infty}\int_{0}^{1}\frac{tdt}{(n+t)^{z_0+1}}-\sum_{n=1}^{\infty}\int_{0}^{1}\frac{tdt}{(n+t)^{z_0+1}}+\frac{1}{z_0-1}-z_0\sum_{n=1}^{\infty}\frac{tdt}{(n+t)^{z_0+1}}}=\\\\
\displaystyle{-\sum_{n=1}^{\infty}\int_{0}^{1}\frac{tdt}{(n+t)^{z_0+1}}+\frac{1}{z_0-1}=\frac{-(z_0-1)\sum_{n=1}^{\infty}\int_{0}^{1}\frac{tdt}{(n+t)^{z_0}}+1}{z_0-1}=\frac{-1+1}{z_0-1}=\frac{0}{z_0-1}=0}\\\\
\displaystyle{\Rightarrow \zeta(z_0)=0}\\\\
\end{array}
$$
\end{proof}

\begin{lema}
For $B=\{z=x+iy | 0<x<\frac{1}{2} \veebar \frac{1}{2} < x < 1 \wedge y\in \mathbb{R}\}$. Then $$\forall \alpha \neq 0 \wedge \forall x:\, 0<x<\frac{1}{2} \veebar \frac{1}{2} < x < 1$$ we have that:
$$1\neq \alpha (x+iy)[(x+iy)-1]$$
\end{lema}

\begin{proof} 
Let's suppose that $\exists \alpha_1 \neq 0 \wedge \exists (x_1+iy)$ such that $0<x_1<\frac{1}{2} \veebar \frac{1}{2}<x_1<1$ and $1=\alpha _1 [(x_1+iy)^2-(x_1+iy)]$ then, derivating with respect to y we find that:
$$0=\alpha_1[2(x_1+iy)i-i]=\alpha_1[2(x_1+iy)-1]i$$ $$\Rightarrow 2(x_1+iy)-1=0$$ $$\Rightarrow (x_1+iy)=\frac{1}{2}$$ $$\Rightarrow x_1+iy=\frac{1}{2}+0i$$ $$\Rightarrow x_1=\frac{1}{2} \,\text{ This is Absurd!}$$

Therefore, $\forall \alpha \neq 0 \wedge \forall x: 0<x<\frac{1}{2} \veebar \frac{1}{2}<x<1$ we have that: $1\neq \alpha(x+iy)[(x+iy)-1]$ in particular if $x_0+iy_0=z_0\in B$, we have $1\neq \alpha z_0(z_0-1)\,\, \forall \alpha \neq 0$
\end{proof}

\begin{teo}[The Riemman's conjeture]
$\forall z \in B : \zeta (z)\neq 0$
\end{teo}

\begin{proof}
Let's suppose that: $\exists z_0 \in B : \zeta(z_0)=0$
 
\begin{equation}\label{tres}
\zeta(z_0)=0 \wedge (Theorem 2) \Rightarrow (z_0-1) \sum_{n=1}^{\infty}\int_{0}^{1}\frac{t dt}{(n+t)^{z_0+1}}=1
\end{equation}
  
\begin{equation}\label{cuatro}
(2) \wedge \zeta(z_0)=0 \Rightarrow \zeta(1-z_0)=0
\end{equation}

\begin{equation}\label{cinco} 
\zeta(1-z_0)=0 \wedge (Theorem 2) \Rightarrow -z_0 \sum_{n=1}^{\infty}\int_{0}^{1}\frac{t dt}{(n+t)^{2-z_0}}=1
\end{equation}

\begin{equation}\label{seis} 
(5) \Rightarrow \sum_{n=1}^{\infty}\int_{0}^{1}\frac{tdt}{(n+t)^{2-z_0}}\neq 0
\end{equation}

$$ \text{Claim!}\,\, \sum_{n=1}^{\infty}\int_{0}^{1}\frac{t dt}{(n+t)^{z_0+1}} \neq -\sum_{n=1}^{\infty}\int_{0}^{1}\frac{t dt}{(n+t)^{2-z_0}}$$

Let's suppose: $$\sum_{n=1}^{\infty}\int_{0}^{1}\frac{t dt}{(n+t)^{z_0+1}} = -\sum_{n=1}^{\infty}\int_{0}^{1}\frac{t dt}{(n+t)^{2-z_0}}$$ 

$$\Rightarrow (z_0-1)\sum_{n=1}^{\infty}\int_{0}^{1}\frac{t dt}{(n+t)^{z_0+1}} = -(z_0-1)\sum_{n=1}^{\infty}\int_{0}^{1}\frac{t dt}{(n+t)^{2-z_0}}$$

\begin{equation}\label{siete}
\Rightarrow (z_0-1)\sum_{n=1}^{\infty}\int_{0}^{1}\frac{t dt}{(n+t)^{z_0+1}} = -z_0\sum_{n=1}^{\infty}\int_{0}^{1}\frac{t dt}{(n+t)^{2-z_0}} + \sum_{n=1}^{\infty}\int_{0}^{1}\frac{t dt}{(n+t)^{2-z_0}}
\end{equation}
 
$$(3)\wedge (5)\wedge (7) \Rightarrow 1 = 1+\sum_{n=1}^{\infty}\int_{0}^{1}\frac{t dt}{(n+t)^{2-z_0}}$$

$$\Rightarrow \sum_{n=1}^{\infty}\int_{0}^{1}\frac{t dt}{(n+t)^{2-z_0}}=0 \,\text{ This is Absurd!}\,\,\, \text{ (By (6))}$$

Then:
\begin{equation}\label{ocho}
\sum_{n=1}^{\infty}\int_{0}^{1}\frac{t dt}{(n+t)^{z_0+1}} \neq -\sum_{n=1}^{\infty}\int_{0}^{1}\frac{t dt}{(n+t)^{2-z_0}}
\end{equation}

\begin{equation}\label{nueve}
(8) \Rightarrow \exists ! \alpha_1 \neq 0,\alpha_1 \in \mathbb{C} \text{ such that } \sum_{n=1}^{\infty}\int_{0}^{1}\frac{t dt}{(n+t)^{z_0+1}} = -\sum_{n=1}^{\infty}\int_{0}^{1}\frac{t dt}{(n+t)^{2-z_0}} + \alpha_1
\end{equation}

$$(9) \Rightarrow (z_0-1)\sum_{n=1}^{\infty}\int_{0}^{1}\frac{t dt}{(n+t)^{z_0+1}}=-(z_0-1)\sum_{n=1}^{\infty}\int_{0}^{1}\frac{t dt}{(n+t)^{2-z_0}} + \alpha_1 (z_0-1)$$

\begin{equation}\label{diez}
\Rightarrow (z_0-1)\sum_{n=1}^{\infty}\int_{0}^{1}\frac{t dt}{(n+t)^{z_0+1}}=-z_0\sum_{n=1}^{\infty}\int_{0}^{1}\frac{t dt}{(n+t)^{2-z_0}} + \sum_{n=1}^{\infty}\int_{0}^{1}\frac{t dt}{(n+t)^{2-z_0}} + \alpha_1 (z_0-1)
\end{equation}  

$$(3)\wedge (5) \wedge (10) \Rightarrow 1=1+\sum_{n=1}^{\infty}\int_{0}^{1}\frac{t dt}{(n+t)^{2-z_0}} + \alpha_1(z_0-1)$$
 
$$\Rightarrow \sum_{n=1}^{\infty}\int_{0}^{1}\frac{t dt}{(n+t)^{2-z_0}} + \alpha_1(z_0-1)=0$$

$$\Rightarrow -z_0\Big[\sum_{n=1}^{\infty}\int_{0}^{1}\frac{t dt}{(n+t)^{2-z_0}} + \alpha_1(z_0-1)\Big]=0$$

\begin{equation}\label{once}
\Rightarrow -z_0\sum_{n=1}^{\infty}\int_{0}^{1}\frac{t dt}{(n+t)^{2-z_0}}-\alpha_1 z_o (z_0-1)=0
\end{equation}

$$(5)\wedge (11) \Rightarrow 1-\alpha_1 z_0 (z_0-1)=0 \Rightarrow 1=\alpha_1 z_0 (z_0-1)\, \, \text{This is Absurd!}\,\text{ (By lemma 2).}$$
Then the proposition ``$\exists z_0\in B:\zeta(z_0)=0$'' is false. Therefore:
$$\forall z \in B: \zeta(z)\neq 0$$

\end{proof} 

\section{Conclusion of the Saga.}
It is known that $\zeta(z)=0$ for some $z\in F$. See for example: [11], [25], [35] or [38].\\

Not every $z\in F$ is solution  for $\zeta(z)=0$, for example $\frac{1}{2}=z_0\in F$
and it is not difficult to prove that $\zeta(z_0)\neq 0$. We can say now that:

$$R=\{z\in \overline{B}: \zeta(z)=0\}=\{z\in F: (z-1)\sum_{n=1}^{\infty}\int_{0}^{1}\frac{tdt}{(n+t)^{z+1}}=1\}$$

Now we know that the non-trivial zeros of $\zeta(z)=0$ are on the critical
line. Therefore:

To find non-trivial solutions for $\zeta(z)=0$;$z\in F$, we can try the system:

\begin{equation*}
\left\{ \begin{array}{rcl}
\displaystyle{(z-1)\sum_{n=1}^{\infty}\int_{0}^{1}\frac{tdt}{(n+t)^{z+1}}} & = & 1\\
z & = & \frac{1}{2}+iy
\end{array}\right.
\end{equation*}

for $n$ big enough could be useful to try the system

\begin{equation*}
\left\{ \begin{array}{rcr}
\displaystyle \int_{0}^{1}\frac{tdt}{(n+t)^{z+1}} & \approx & \frac{1}{n(n+1)}\\
z & = & \frac{1}{2}+iy
\end{array}\right.
\end{equation*}

\begin{coro} If every statement of the type ``$RH\Leftrightarrow A=B$'' is true. Then
$A=B$.

Where $A=B$ means a relation betwen $A$ and $B$. See below.
\end{coro}

\begin{proof} $\text{ }$\\

That ``RH'' is true follows from Theorem 3. Then:\\

\noindent $\text{``}RH\text{''} \Rightarrow \text{``}A=B \text{''}$\\
$\underline{\text{``}RH\text{''} \text{ }\text{ }\text{ }\text{ }\text{ }\text{ }\text{ }\text{ }\text{ }\text{ }\text{ }\text{ }\text{ }\text{ }\text{ }\text{ }}$

\hspace{0.2cm} $\text{``}A=B \text{''}$\\

See \cite{4} pags. 13--16.
\end{proof}

\section{Applications.}

\begin{description}
    \item[$1)$] Redheffer (1977)
    \[R.H.\Leftrightarrow\forall\,\varepsilon>0, \exists\,\,C(\varepsilon)>0 :|\det
    {(A(n))}|<C(\varepsilon)n^{
    \frac{1}{2}+\varepsilon}\]
    where $A(n)$ is the $n\times n$ matrix of $0$'s and $1$'s defined by
    \[A(i,j)=\begin{cases}
    1&\hbox{ if } j=1\hbox{ or if } i|j\\
    0&\hbox{ otherwise }
    \end{cases}\]
    This is an important result for linear analysis for example. See \cite{10.2}
    \item[$2)$]Lagarias (2002)

    Let $\sigma(n)$ denotes the sum of the positive divisors of $n$. Then
    \[R.H.\Leftrightarrow\forall\,n:\,\sigma(n)\le H_n+\exp{(H_n)}\log{H_n}\]
    where $H_n=1+\frac{1}{2}+\frac{1}{3}+\cdots+\frac{1}{n}$. This is an important
    result for
    number theory for example. See \cite{10.2}.
    \item[$3)$] Nyman-Beurling
    \[R.H.\Leftrightarrow Span_{L^2(0,1)}=\{\mathcal{N}_\alpha:0<\alpha<1\}=L^2(0,1)
    \]
    where
    \[\mathcal{N}_\alpha(t)=\left\{\frac{\alpha}{t}\right\}-\alpha\left\{\frac{1}{t}\right\}\]
    and $\{x\}=x-[x]$ is the fractional part of $x$ this is an important result for
    Real and Functional
    Analysis for example. See \cite{10.2}
\end{description}

Others results like these can be seen in \cite{10.2}. I believe that one of the must important result
to be studied after this one is the paper of Andre Weil. See \cite{10.2}.

\begin{coro}
If $\chi=\chi_{_{1}}$ is the principal character $mod\, k$ then: $$\forall s\in B: \, L(s,\chi_{_{1}})\neq 0$$
\end{coro}
\begin{proof}
See [2] Theorem 11.7 and then use theorem 3.
\end{proof}

\section{Open Questions}
\begin{enumerate}
    \item Are simple the zeros of the $\zeta$ Riemann Function?
    \item It is known that of all imaginary quadratic field $Q(\sqrt{-d})$ with
    class number $h$, we have
    $d<Ch^2\log{h}^2$, except for at most one exceptional field, for which $d$ may
    be Larger.

    Does there exist such an Exceptional Field?
\end{enumerate}

Hint.- See \cite{22.1} and \cite{22.2}.

\noindent Prof: Pedro J, Geraldo C.\\
Home address:\\
Calle principal 130. Delicias Nuevas\\
Cabimas (4013). Edo Zulia Venezuela.\\
E-mail: pegeraldo@luz.edu.ve \& pegeraldo@yahoo.com\\
Home phone: 0264-2513221

\begin{thebibliography}{99}
\bibitem{1} L.V. Alphors. {\em Complex Analysis}. Mc Graw Hill Book Company. Second
Edition. Tokyo
1966.
\bibitem{2} T. M. Apostol. {\em Introduction to Analytic Number Theory}. Springer-
Verlag. New York Inc.
1980.
\bibitem{3.1} R. Bellman. {\em A Brief Introduction to Theta function}. Holt Rinehart
and Wistons. USA
1961.
\bibitem{3.2} R. Bellman. {\em A Collection of Modern Mathematical Classics}. Analysis.
Dover. New York
1961.
\bibitem{4} M.L. Bittinger. {\em Proof, Logic an Sets}. Addison Wesley Publishing
Company. Reading
Massachusetts. USA. 1982.
\bibitem{5} E. Bombieri. {\em Problems of the Millenium the Riemann Hypothesis}.

{\url{http://www.claymath.org/prizeproblems/riemann.htm}}.
\bibitem{6} H. Cartan. {\em Theory elementaire des Fonctions Analytiques D'une ou
Plusieurs Variables
Complexes}. Hermann Editeurs Des Sciences Et Des Arts. Paris. 1985.
\bibitem{7} K. Chandrasekharan. {\em Introduction to Analytic Number Theory}. Springer
Verlag. New
York. 1968.
\bibitem{8} L. Chambadal. {\em Dictionarie Des Mathematiques Modernes}. Libraie
Larousse. Paris. 1969.
\bibitem{9} H. Cohn. {\em Advanced Number Theory}. Dover Publishing, Inc. New York.
1962.
\bibitem{10.1} J.B. Conrey. {\em More Than Two Fifths of the zeros of the Riemann Zeta
Functions are on
the Critical Line}. J. Reine Angew. MAth 399 (1989) 1-26.
\bibitem{10.2} J.B. Conrey. {\em The Riemann Hypothesis}. Notices of the American
Mathematical Society.
Volumn 50. Number 3. (March 2003) 341353.
\bibitem{11} J.B. Conway. {\em Function of one Complex Variable}. Springer Verlag. N.Y.
1973.
\bibitem{12} L.E. Dickson. {\em History of the Theory of Numbers}. Chelsea New York.
1952.
\bibitem{13} M.H. Edwards. {\em Riemann's Zeta Function}. Academic Press, New York-
London. 1974.
\bibitem{13.2} L. Flatto. Advanced Calculus. The Williams and Wilkins Company.
Baltimore 1976. USA.
\bibitem{13.3} E. Gentile, {\em Notas de Algebra}. EDEBA. Buenos Aires. 1976.
\bibitem{14} HArdy, G.H. and Wright, E.M. {\em An Introduction to the Theory of
Numbers}. $4^{th}$ Ed.
Clarendon Press. Oxford. 1960.
\bibitem{15} S.T. Hu. {\em Introduction to General Topology}. Holden-Day, Inc. San
Francisco. USA 1966.
\bibitem{16} A.A. Karatsuba. {\em Fundamentos de la Teor\'ia anal\'itica de los N\'umeros}. Editorial MIR.
Moscu. 1979.
\bibitem{17} K. Knopp. {\em Theory of Functions}. Parts I and II. Dover Publications.
New York. 1947.
\bibitem{18.1} S. Lang. {\em Complex Analysis}. Addison Wesley. Reading Mass. USA. 1976.
\bibitem{18.2} {\em Theory of Numbers}. Spriger-Verlag. New York. USA 1980
\bibitem{19} D. Laugwitz, A. Sheinitzer. {\em Bernhard Riemann 1826-1866}. Birkhauser.
Boston. Baser.
Berlin 1998.
\bibitem{20} N. Levison. {\em More than One third of zeros of Riemann's Zeta Function
are on
$\sigma=\frac{1}{2}$}. Advances Math. 13. 383-486 (1974).
\bibitem{22} S. MacLane. {\em Symbolic Logic}. American Mathematical Montly. Vol. 46. P.
289. (1939).
\bibitem{21} A. Markushevich. {\em Teor\'ia de las Funciones Anal\'iticas.}.  Tomos I \&
II. Editorial MIR.
Mosc\'u 1970.
\bibitem{22.1} H.L. Montgomery, {\em The pair correlation of zeros of the zeta
function}, Proc. sympos. Pure Math., vol.
24, amer. Math. soc., Providence, R.I., 1973, pp. 181-193.
\bibitem{22.2} H.L. Montgomery, and P.J. Weinberger, {\em Notes on small class numbers},
Acta Arith. 24 (1974), 329-342.
\bibitem{22.3} H.L. Montgomery, {\em Distribution of Zeros of the Riemann Zeta
Function}, Proc. Intr. Congrss of Math.
Vancouver, 1974, pp. 379-381.
\bibitem{23} R. Narashiman, Y. Nievergelt. {\em Complex Analysis in one Variable}.
Second Edition.
Birkhauser, Boston 2000.
\bibitem{24} A.M. Odlyzko. {\url{http://www.dtc.umn.edu/~odlizko/}}.
\bibitem{25} W. Rudin. {\em Real and Complex Analysis}. Second Edition. Mc. Graw Hill.
Series in Higher
Math.
\bibitem{25.1} I.E.Segal \& R.A.Kunze. {\em Integrals and operators}. Mc. Graw Hill. company
N.Y. 1968.

\bibitem{26} A. Selberg. {\em On the zeros of the Zeta Function of Riemann}. College
Papers. Springer
Verlag. New York. 1989 Vol I, 156-159.
\bibitem{27} A. Selberg. {\em Old and New Conjetures and Results About A Class of
Dirichlet Series}.
Vol II. With A Foreword By K. Chandrasekharan.
\bibitem{28} C.L. Siegel. {\em Analytic Number Theory}. (Lectures Notes By B. Riemann)
New York
University 1945.
\bibitem{29} E.C. Titchmarsh. {\em The theory of the Riemann Zeta-Function}. Claredon
Press. Oxford
1951.
\bibitem{30} I.Vinogradov. {\em Fundamentos de la Teor\'ia de los N\'umeros}. Editorial
MIR. Mosc\'u. 1971.\\
\end{thebibliography}
\end{document}